\documentclass[a4paper,10pt,reqno, english]{amsart}

\usepackage{amsmath,amssymb,amscd,amsthm,amsfonts}
\usepackage{graphicx,subfigure}
\usepackage{hyperref}
\usepackage{dsfont}

\newtheorem{theorem}{Theorem}[section]
\newtheorem*{theoremp}{Theorem}
\newtheorem{lemma}[theorem]{Lemma}
\newtheorem{claim}[theorem]{Claim}

\newtheorem{question}[theorem]{Open question}

\newtheorem{definition}{Definition}

\def\rr{\mathds{R}}

\DeclareMathOperator{\conv}{conv}

\title{Balanced convex partitions of lines in the plane}

\author[Xue]{Alexander Xue}
\address{Cornell University, Ithaca, NY 14850, United States} 
\email{ajx3@cornell.edu}

\author[Sober\'on]{Pablo Sober\'on}\address{Baruch College, City University of New York, One Bernard Baruch Way, New York, NY 10010, United States} 
\email{pablo.soberon-bravo@baruch.cuny.edu}

\thanks{
This research project was done as
part of the 2019 CUNY Combinatorics REU, supported by NSF awards
DMS-1802059 and DMS-1851420.}

\begin{document}

\maketitle

\begin{abstract}
We prove an extension of a ham sandwich theorem for families of lines in the plane by Dujmovi\'{c} and Langerman.  Given two sets $A, B$ of $n$ lines each in the plane, we prove that it is possible to partition the plane into $r$ convex regions such that the following holds.  For each region $C$ of the partition there is a subset of $c_r n^{1/r}$ lines of $A$ whose pairwise intersections are in $C$, and the same holds for $B$.  In this statement $c_r$ only depends on $r$.  We also prove that the dependence on $n$ is optimal. 
\end{abstract}


\section{Introduction}

A general measure partition problem deals with the way we can split points or measures in Euclidean spaces.  Given a set of rules to split the ambient space, we are interested to know if we can divide a given set of points in a prescribed way.  The quintessential result of this kind is the classic ham sandwich theorem.

\begin{theoremp}
Given $d$ finite sets of points in $\rr^d$ in general position such that each set is of even cardinality there exists a hyperplane that simultaneously splits each set exactly by half.	
\end{theoremp}

The proof of a mass partition result usually boils down to understanding topological properties of the space of partitions \cite{Zivaljevic:2004vi}.  The methods developed to tackle measure partitions problems have broad applications in combinatorial topology.

In this manuscript we are interested in extensions of the ham sandwich theorem for convex partitions of the plane.  A convex partition of $\rr^2$ into $r$ parts is a family of closed sets $C_1, \ldots, C_r \subset \rr^2$ such that
\begin{itemize}
	\item the sets cover $\rr^2$, so $\cup_{i=1}^r C_i = \rr^2$,
	\item the interiors of the sets are pairwise disjoint, and
	\item each $C_i$ is a closed convex set.
\end{itemize}

The ham sandwich theorem has been generalized to convex partitions of the plane.  The following theorem was proven independently by Ito, Uehara, and Yokoyama \cite{Ito:2000eb}, by Bespamyatnikh, Kirkpatrick, and Snoeyink \cite{Bespamyatnikh:2000tn}, and by Sakai \cite{Sakai:2002vs}.

\begin{theorem}\label{theorem-equipartition-of-points}
Let $A_1, A_2$ be two finite sets of points in $\rr^2$, in general position.  If the cardinality of each set is a multiple of $r$, there exists a partition of the plane into $r$ convex sets $C_1, \ldots, C_r$ such that for each $i \in \{1,2\}, j \in \{1,2,\ldots, r\}$ we have
\[
|A_i \cap C_j| = \frac{1}{r}|A_i|.
\]
In other words, each set is partitioned evenly.
\end{theorem}

The continuous version of the theorem above has been generalized to convex partitions of $\rr^d$ with $d$ measures \cite{Soberon:2012kp, Karasev:2014gi, Blagojevic:2014ey}.  The high-dimensional versions of Karasev, Hubard, and Aronov \cite{Karasev:2014gi}, and of Blagojevi\'{c} and Ziegler \cite{Blagojevic:2014ey} hold in a much more general setting.  These were motivated by a problem of Nandakumar and Rao, which has been recently solved \cite{Akopyan:2018tr}.  A discrete version in high dimensions was recently established \cite{Blagojevic:2019gb}.  Theorem \ref{theorem-equipartition-of-points} can be bootstrapped to obtain partitions of measures where each part has positive size in an arbitrary number of measures \cite{BPSZ17}. The planar version has applications to drawings of political district maps \cite{Humphreys:2011iy, Soberon:2017kt}.

In this manuscript we are interested in splitting families of lines in $\rr^2$ instead of families of points.  Our main result is an extension of Theorem \ref{theorem-equipartition-of-points} to families of lines.  We use the following definition for what it means to split a family of lines.  Given a set $L$ of lines in the plane such that no two lines of $L$ are parallel, let $I(L)$ be the set of pairwise intersection points of $L$.  We call $I(L)$ the \textbf{incidence set} of $L$.

\begin{definition}
	Given a closed set $K \subset \rr^d$, and a set of lines $L$ in the plane such that no two sets of lines of $L$ are parallel, we say that $K$ \textbf{encloses} $L$ if 
	\[
	I(L) \subset K.
	\]
\end{definition}  

If $K$ is convex, the condition above is equivalent to $\conv( I(L)) \subset K$.  A ham sandwich theorem with this definition was proved by Dujmovi\'{c} and Langerman.

\begin{theorem}[Dujmovi\'{c}, Langerman 2013 \cite{Dujmovic:2013bi}]\label{theorem-dujmovic}\label{theorem-twoparts-for-lines}
	Given two finite sets $A$, $B$ of lines each in the plane, if no two lines of $A \cup B$ are parallel, there exists a line $\ell$ such that each of the two closed half-spaces it defines encloses a subset of at least $\sqrt{|A|}$ lines of $A$ and a subset of at least $\sqrt{|B|}$ lines of $B$.
\end{theorem}

  Given two sets of lines, we obtain convex partitions of the plane where each part encloses a large subset of lines.  Partitions related to families of lines have been studied before in other settings.  For instance, the celebrated polynomial partitioning method of Guth and Katz shows the existence of an equipartition of $I(L)$ using a low-degree polynomial, where no line intersects too many regions \cite{Guth:2015tu, Guth:2016wo}.  A recent work of Schnider proves extensions of the ham sandwich theorem for families of lines in $\rr^3$ \cite{Schnider:2019tg}, under a different interpretation of separation of lines.  In order to state our main result, we need the following definition.

\begin{definition}
	We say that a set of lines $L$ in the plane is in general position if
	\begin{itemize}
		\item No two lines in $L$ are parallel
		\item No three lines in $L$ are concurrent
		\item If three points in $I(L)$ are colinear, they belong to the same line in $L$.
	\end{itemize}
\end{definition}

With this, our main result is the following.

\begin{theorem}\label{theorem-main}
	Let $A, B$ be two finite sets of lines in $\rr^2$ such that $A \cup B$ is in general position, and let $r$ be a fixed positive integer.  Then, there is a convex partition $(C_1, \ldots, C_r)$ of $\rr^2$ into $r$ parts such that for all $j \in \{1, \ldots, r\}$ there exist sets $A_j \subset A$, $B_j \subset B$ such that $I(A_j) \subset C_j, I(B_j) \subset C_j$ and
	\[
	|A_j| \ge r^{\ln (2/3)} |A|^{1/r}-2r, \qquad 	|B_j| \ge r^{\ln (2/3)} |B|^{1/r}-2r.
	\]
\end{theorem}

  Notice that, if $r$ is a power of two, a repeated application of Theorem \ref{theorem-dujmovic} implies the existence of a convex partition into $r$ parts where each part encloses $|A|^{1/r}$ lines of $A$ and $|B|^{1/r}$ lines of $B$.  We do not know if the leading factor  $r^{\ln (2/3)} \sim r^{-0.405}$ is necessary in general.  The dependence of Theorem \ref{theorem-main} on $|A|$ and $|B|$ is optimal up to that factor, as the following theorem shows.

\begin{theorem}\label{theorem-upper-bound}
  Let $n, r$ be positive integers.  There exist sets $A, B$ of $n$ lines in the plane in general position such that for every convex partition $C_1, \ldots, C_r$ of the plane into $r$ sets, there exists some $j \in \{1, \ldots, r \}$ such that either $C_j$ encloses at most $\lceil n^{1/r} \rceil$ lines of $A$ or $C_j$ encloses at most $\lceil n^{1/r} \rceil$ lines of $B$.	
\end{theorem}


The proof of Theorem \ref{theorem-main} is similar to the first proof of Theorem \ref{theorem-equipartition-of-points} \cite{Bespamyatnikh:2000tn}.  The main idea is to use only partitions of $\rr^2$ into two or three convex pieces, and subdivide each piece until we obtain the desired partition. The only topological tool needed is the Knaster-Kuratowski-Mazurkiewicz theorem \cite{Knaster:1929vi} for a triangle.  In Section \ref{section-partition-into-three-parts} we prove an Erd\H{o}s-Szekeres type lemma which is crucial in the main proof.  In Sections \ref{section-equitable-cuttings} and \ref{section-region-construction} we discuss properties of partitions of $\rr^2$ into three parts.  In Section \ref{section-big-proof} we prove Theorem \ref{theorem-main}, and we prove our upper bounds in Section \ref{section-upper-bounds}. Finally, in Section \ref{section-remarks} we include remarks in the proof.  We discuss the extensions of the Erd\H{o}s-Szekeres theorem which would allow us to extend other proofs of Theorem \ref{theorem-equipartition-of-points} to the setting with lines.

\section{Partitions into two and three parts}\label{section-partition-into-three-parts}

Given a set $A$ of lines in the plane, and $K \subset \rr^d$, we consider 

\[
\mu_A(K) = \max\{|A'|: A' \subset A, I(A') \subset K \}.
\]

We now prove a couple of properties of $\mu_A$ for convex partitions of $\rr^2$ into two or three parts.

\begin{lemma}\label{lemma-two-parts}
	Let $A$ be a set of $n$ lines in $\rr^2$ in general position.  Let $\ell$ be another line in the plane, which defines a convex partition of $\rr^2$ into two parts $C_1, C_2$.  Then we have
	\[
	\mu_A(C_1) \mu_A(C_2) \ge n.
	\]
\end{lemma}

\begin{proof}
	First, assume that $\ell$ does not contain any point of $I(A)$ and is not parallel to any line in $\ell$.  We may assume that $\ell$ is a vertical line.  We order the points of the form $\ell \cap a$ with lines in $a \in A$ by their vertical coordinates, from bottom to top.  Suppose that $\ell \cap a$ is the $i$-th point.  We define $x_i$ to be the slope of $a$.  Consider the sequence $(x_1, \ldots, x_n)$.  An increasing subsequence defines a subset of $A$ enclosed by the left side of $\ell$, while a decreasing subsequence defines a subset of $A$ enclosed by the right side of $\ell$.  Therefore, the Erd\H{o}s-Szekeres theorem finishes this case.
	
	If $\ell \in A$, then we apply the argument above to $A \setminus \ell$ and add $\ell$ to each enclosed set.  If $\ell \not\in A$ and it either contains a point of $I(A)$, is parallel to a line in $A$, or both, a standard approximation argument finishes the proof.
\end{proof}

\begin{lemma}\label{lemma-xue}
	Let $A$ be a set of $n$ lines in $\rr^2$ in general position.  Let $r_1$, $r_2$ be two rays that star from a common point $p$.  The broken line $r_1 \cup r_2$ splits $\rr^2$ into two sets $C_1, C_2$.  Then,
	\[
	\mu_A(C_1) \mu_A (C_2) \ge \frac{2n}{3}.
	\]
\end{lemma}

\begin{proof}
	We assume that $r_1$ and $r_2$ are not contained in lines of $A$, that they do not contain points of $I(A)$ and that they are not parallel to lines of $A$.  A standard approximation argument shows that we do not lose generality by making these assumptions.  We may also assume without loss of generality that $r_2$ is the positive $y$-axis, $r_2$ form an acute angle with the positive $x$-axis, and $C_1$ is convex.  See Figure \ref{figure-xue}.  We split the lines in $A$ into four types:
	\begin{itemize}
		\item Type 1: lines intersecting $r_1$ but not $r_2$
		\item Type 2: lines intersecting $r_2$ but not $r_1$
		\item Type 3: lines intersecting both $r_1$ and $r_2$
		\item Type 4: lines intersecting neither $r_1$ nor $r_2$.
	\end{itemize}
	
	\begin{figure}
		\centerline{\includegraphics{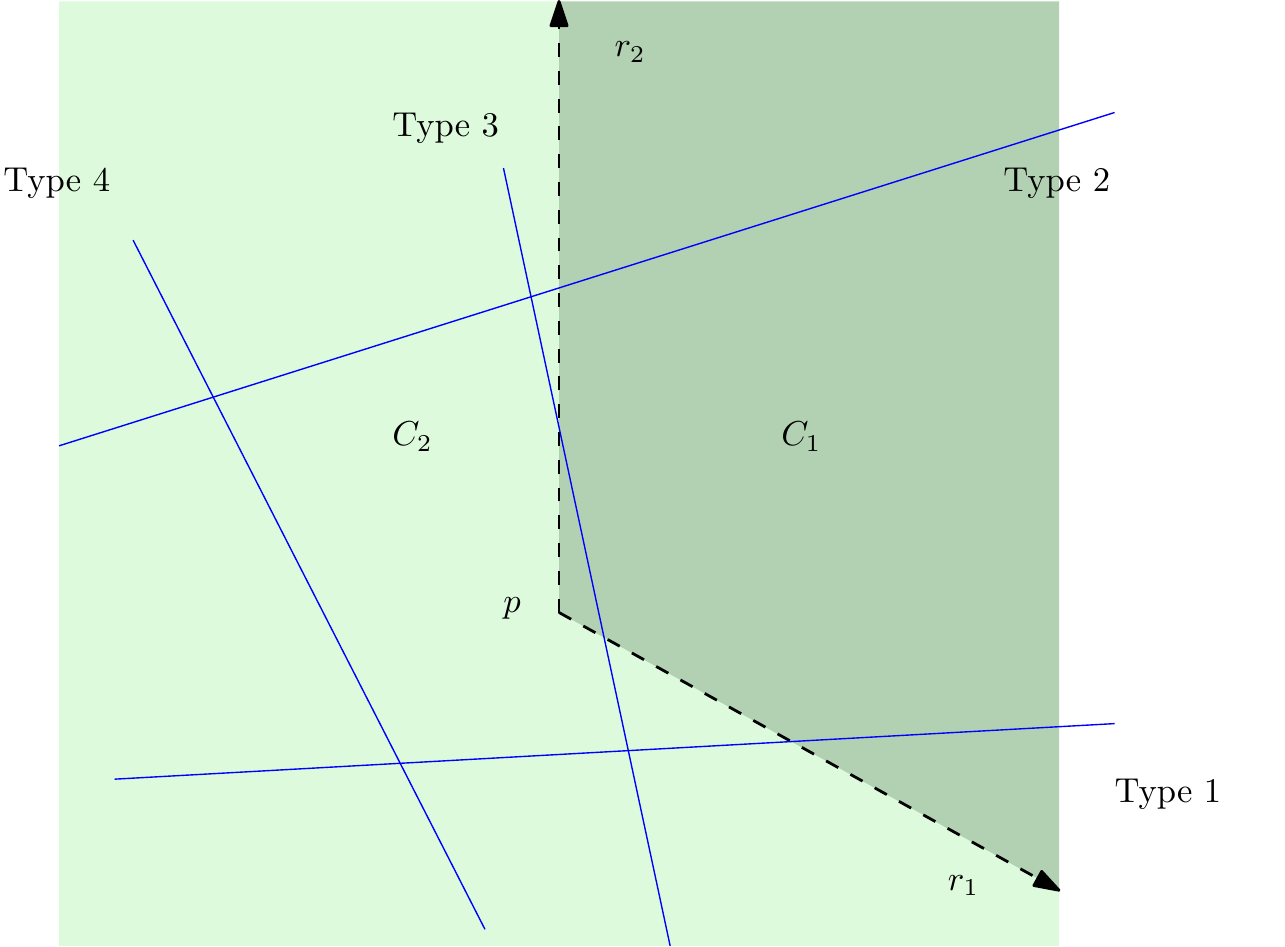}}
		\caption{Type assignment for the proof of Lemma \ref{lemma-xue}}
		\label{figure-xue}
	\end{figure}
	
	Then, we consider the sets 
	\begin{align*}
		A_1 & = \{ a \in A: a \mbox{ has type 2,3, or 4}\} \\
		A_2 & = \{ a \in A: a \mbox{ has type 1,3, or 4}\} \\
		A_3 & = \{ a \in A: a \mbox{ has type 1 or 2}\}
	\end{align*}
	
	Notice that $|A_1|+|A_2|+|A_3| = 2n$, so at least one of the sets has cardinality greater than or equal to $\frac{2n}{3}$.  We are going to define a partial order on each of the three sets in such a way that the elements of any chain pairwise intersect in $C_1$ and the element of any anti-chain pairwise intersect in $C_2$.  An application of Dilworth's theorem to the largest set among $A_1, A_2, A_3$ will give us the desired conclusion.
	
	Let us start with $A_1$.  For a line $\ell$ that intersects $r_2$, we define
	\begin{align*}
	d(\ell) & = \operatorname{dist}(p, r_2 \cap \ell). \\
	s(\ell) &= \mbox{ the slope of }\ell.	
	\end{align*}

	We say that $\ell_1$ is comparable to $\ell_2$ if either $\ell_1 = \ell_2$ or $\ell_1 \cap \ell_2 \in C_1$.  Furthermore, if $\ell_1$ and $\ell_2$ are comparable, we say that $\ell_1 \le \ell_2$ if either
	
	\begin{itemize}
		\item $\ell_1$ is of type $2$ and $\ell_2$ is of type $3$ or
		\item $\ell_1$ and $\ell_2$ are of the same type and $d(\ell_1) \le d(\ell_2)$.
	\end{itemize}
	
	In order to show that this is a partial order, we prove it is transitive.  Suppose that $\ell_1 \le \ell_2$ and $\ell_2 \le \ell_3$.
	
	\bigskip
	
	\textbf{Case 1.} The lines $\ell_1, \ell_2, \ell_3$ are all of type $2$.  In this case, we know $d(\ell_1) \le d(\ell_2)$ and $d(\ell_2) \le d(\ell_3)$, so $d(\ell_1) \le d(\ell_3)$.  We just need to know that $\ell_1 \cap \ell_3 \in C_1$.  Since $d(\ell_1) \le d(\ell_2)$ and $\ell_1 \cap \ell_2$ has a positive $x$-coordinate, then $s(\ell_1) \ge s(\ell_2)$.  Similarly, $s(\ell_2) \ge s(\ell_3)$, so $s(\ell_1) \ge s(\ell_3)$.  Therefore, $\ell_1 \cap \ell_3$ has a positive $x$-coordinate.  Since neither $\ell_1$ nor $\ell_3$ intersect $r_1$, their intersection is is $C_1$.
	
	\textbf{Case 2.} The lines $\ell_1, \ell_2$ are of type $2$ and $\ell_3$ is of type $3$.  It suffices to show that $\ell_1 \cap \ell_3 \in C_1$.  This is equivalent to showing that $d(\ell_1) \le d(\ell_3)$.  However, $d(\ell_1) \le d(\ell_2)$ and $d(\ell_2) \le d(\ell_3)$, so we obtain the desired inequality.
	
	\textbf{Case 3.}  The line $\ell_1$ is of type $2$ and the lines $\ell_2, \ell_3$ are of type $3$.  We just need to show that $\ell_1 \cap \ell_3 \in C_1$.  This reduces to showing $d(\ell_1) \le d(\ell_3)$, which can be done as in Case 2.
	
	\textbf{Case 4.}  The lines $\ell_1, \ell_2, \ell_3$ are all of type $3$.  We obtain $d(\ell_1) \le d(\ell_3)$ as case $1$, so we just need to show that $\ell_1 \cap \ell_3 \in C_1$.  This reduces to showing that the intersections $\ell_1 \cap r_1, \ell_3 \cap r_1$ appear in reverse order as $\ell_1 \cap r_2, \ell_3 \cap r_2$.  However, since this happens for the pairs of lines $(\ell_1, \ell_2)$ and $(\ell_2, \ell_3)$, we are done.
	
	\bigskip
	
	The order for $A_2$ can be defined in an analogous way, replacing the roles of $r_1$ and $r_2$.
	
	Let us define the order of $A_3$.  If $\ell$ is a line that intersects $r_1$, we define
	
	\[
	d'(\ell) = \operatorname{dist} (p, \ell \cap r_1)
	\]
	
	We say that two lines $\ell_1, \ell_2$ in $A_3$ are comparable if $\ell_1 \cap \ell_2 \in C_1$.  Moreover, if they are comparable we say that $\ell_1 \le \ell_2$ if either
	\begin{itemize}
		\item The lines $\ell_1, \ell_2$ are both of type $2$ and $d(\ell_1) \le d(\ell_2)$,
		\item The lines $\ell_1, \ell_2$ are both of type $1$ and $d'(\ell_1) \ge d'(\ell_2)$, or
		\item The line $\ell_1$ is of type $1$ and the line $\ell_2$ is of type $2$.
	\end{itemize}
	
	Let us prove transitivity.  If $\ell_1 \le \ell_2$ and $\ell_2 \le \ell_3$, the work above shows that $\ell_1 \le \ell_3$ if the three lines are of the same type.  Therefore, we only need to check two cases, in which the line $\ell_1$ is of type $1$ and the line $\ell_3$ is of type $2$.
	
	\bigskip
	
	\textbf{Case 1.}  The line $\ell_2$ is of type $2$.  Because $\ell_1 \le \ell_2$, we have that $s(\ell_1) \ge s(\ell_2)$.  Since $\ell_2 \cap \ell_3 \in C_1$, we have that $s(\ell_2) \ge s(\ell_3)$.  Therefore, $s(\ell_1) \ge s(\ell_3)$.  This means that $\ell_1 \cap \ell_3 \in C_1$.
	
	\textbf{Case 2.}  The line $\ell_2$ is of type $1$.  This case is analogous to Case 1 if we swap the roles of $r_1$ and $r_2$ and reverse the order.

\end{proof}

\begin{lemma} \label{lemma-xue2}
	Let $A$ be a set of $n$ lines in $\rr^2$ in general position.  Let $(C_1, C_2, C_3)$ be a convex partition of $\rr^2$ into three convex parts.  Then,
	
	\[
	\mu_A (C_1) \mu_A(C_2) \mu_A (C_3) \ge \frac{2n}{3}
	\]
\end{lemma}

\begin{proof}
	If $(C_1, C_2, C_3)$ is formed by two parallel lines, we can apply Lemma \ref{lemma-two-parts} twice to finish.  If not, then $(C_1, C_2, C_3)$ is formed by three rays coming out of the same point.  Consider the partition $(C_1, C_2 \cup C_3)$.  By Lemma \ref{lemma-xue}, we know that
	\[
	\mu_A (C_1) \mu_A (C_2 \cup C_3) \ge \frac{2n}{3}.
	\]
	
	Let $A' \subset A$ be the set of lines that realizes $\mu_A(C_2 \cup C_3)$.  Let $\ell$ be the line that contains the ray spitting $C_2$ and $C_3$.  We can apply Lemma \ref{lemma-two-parts} to $A'$ and $\ell$ to obtain
	\[
	\mu_A(C_2) \mu_A(C_3) \ge \mu_{A'}(C_2) \mu_{A'} (C_3) \ge |A'| = \mu_{A}(C_2 \cup C_3),
	\]
	which concludes the proof.
\end{proof}

\section{Equitable cuttings.}\label{section-equitable-cuttings}


The gist of the proof for Theorem \ref{theorem-main} is to show that for each value of $r$ and two finite sets of lines $A, B$, each in general position, we always have at least one of the two following situations.

\begin{itemize}
	\item There exists two positive integers $r_1, r_2$ such that $r_1 + r_2 = r$ and there is a convex partition of the plane into two parts $(C_1, C_2)$ such that
	\[
	\mu_A (C_i) \ge \left(\frac{2|A|}{3}\right)^{r_i/r}-2, \quad \mu_B (C_i) \ge \left(\frac{2|B|}{3}\right)^{r_i/r}-2 \qquad \mbox{ for } i=1,2
	\]
	\item There exists three positive integers $r_1, r_2, r_3$ such that $r_1 + r_2 + r_3 = r$ and there is a convex partition of the plane into three parts $(C_1, C_2, C_3)$ such that
	\[
	\mu_A (C_i) \ge \left(\frac{2|A|}{3}\right)^{r_i/r}-2, \quad \mu_B (C_i) \ge \left(\frac{2|B|}{3}\right)^{r_i/r}-2 \qquad \mbox{ for } i=1,2,3
	\]
\end{itemize}

The constant $2/3$ factors are the reason why we have the $r^{\ln(2/3)}$ factor in the main theorem.  We will call the first type of partition an \textbf{equitable $(r_1,r_2)$ cut}, and the second type an \textbf{equitable $(r_1, r_2, r_3)$ cut}.

The rest of the paper will focus on proving the following lemma.

\begin{lemma}\label{lemma-strong-equitable-cuts}
	Let $A, B$ be two finite sets of points in the plane, each in general position, and $r\ge 2$ be a positive integer.  Then, there either exists a pair $(r_1, r_2)$ of positive integers with sum $r$ for which there is an equitable $(r_1, r_2)$ cut, or there exists a triple $(r_1, r_2, r_3)$ of positive integers with sum $r$ for which there is an equitable $(r_1, r_2, r_3)$ cut.
\end{lemma}

Let us first show that Lemma \ref{lemma-strong-equitable-cuts} implies Theorem \ref{theorem-main}.

\begin{proof}[Proof that Lemma \ref{lemma-strong-equitable-cuts} implies Theorem \ref{theorem-main}]
	First, notice that $r^{\ln (2/3)} = \left( \frac23 \right)^{\ln r}$.  We prove Theorem \ref{theorem-main} by strong induction on $r$.  When $r=1$, the result is clear.  If $r \ge 2$, suppose that there is a pair $(r_1, r_2)$ with sum $r$ for which there is an equitable $(r_1, r_2)$ cut.  The case for a triple will be analogous.  Denote by $(C_1, C_2)$ the two sets of the partition.
	
	We know that $\mu_A (C_i) \ge \left( \frac{2|A|}{3}\right)^{r_i/r}-2$.  Let $A_i \subset A$ of cardinality $\mu_A(C_i)$ be enclosed by $C_i$.  We apply Theorem \ref{theorem-main} to $A_i$ to find a convex partition $(C^{i}_1, \ldots, C^{i}_{r_i})$ of the plane into $r_i$ parts.  We know that
	\begin{align*}
		\mu_A(C^{i}_j \cap C_i) & \ge \mu_{A_i}(C'_j \cap C_i) \ge \left( \frac{2}{3}\right)^{\ln (r_i)}|A_i|^{1/r_i}-2r_i \ge \left( \frac23\right)^{\ln (r_i)} \left( \left[\frac{2|A|}{3}\right]^{r_i/r}-2\right)^{1/r_i} - 2r_i \\
		& \ge \left( \frac23\right)^{\ln (r_i)} \left( \left[\frac{2|A|}{3}\right]^{r_i/r}\right)^{1/r_i} - 2(r_i +1) \ge \left( \frac23\right)^{\ln (r-1)+1/r} {|A|}^{1/r} - 2r\\
		 & \ge \left( \frac23\right)^{\ln r}|A|^{1/r}-2r,
	\end{align*}
	where the last inequality follows from the mean value theorem.  Equivalent arguments work for $B$.  Therefore, the partition formed by the sets $C^i_j \cap C_i$ for $i=1,2$, $j =1,\ldots, r_i$ is the one we are looking for.
\end{proof}

\begin{definition}
	Let $r_1$ be an integer with $1 \le r_1 \le r-1$.  We say that a closed half-plane $H$ is $r_1$-critical for $A$ if
	\begin{itemize}
		\item $H$ encloses a subset of $A$ of cardinality at least $\left(\frac{2|A|}{3}\right)^{r_1/r}-2$ and
		\item The interior of $H$ does not enclose a subset of $A$ of cardinality at least at least $\left(\frac{2|A|}{3}\right)^{r_1/r}-2$.
	\end{itemize}
\end{definition}

Notice that the boundary line of an $r_1$-critical halfspace is a support line of a set of the form $\conv(I(A'))$, where $A' \subset A$ and $|A'| = \left\lceil \left(\frac{2|A|}{3}\right)^{r_1/r} \right\rceil-2$, and its interior contains no other such set.

\begin{lemma}
	Let $A$ be a finite family of lines in $\rr^2$ in general position and $r_1$, $r$ be two positive integers with $1 \le r_1 \le r-1$.  Suppose that a line $\ell$ induces a convex partition into two closed half-spaces $(C_1,C_2)$.  If $C_1$ is $r_1$-critical, then
	\[
	\mu_A(C_2) > \left( \frac{3}{2}\right)^{(r_1/r)}|A|^{1-(r_1/r)}
	\]
\end{lemma}

\begin{proof}
	Let $(C'_1, C'_2)$ be a convex partition of the plane such that $C'_1 \subset C_1$ and $C'_1$ contains the same points of $I(A)$ as the interior of $C_1$.  Therefore, $\mu_A (C'_1) < \left(\frac{2|A|}{3}\right)^{r_1/r}$ and $\mu_A(C_2) = \mu_A(C'_2)$.  However, by Lemma \ref{lemma-two-parts}, $\mu_A(C'_1)\mu_A(C'_2) \ge |A|$, so we get the desired conclusion. 
\end{proof}

\begin{lemma}\label{lemma-twocutting}
	Let $A, B$ be two finite sets of lines in the plane, each in general position, and $r\ge 2$ be a positive integer.  Let $1 \le r_1 \le r-1$ be an integer.  If there are two $r_1$-critical half-spaces $H$ and $H'$ for $A$, such that $\mu_B(H) \ge \left(\frac{2|B|}{3}\right)^{r_1/r} -2$ and $\mu_B (H') \le \left(\frac{2|B|}{3}\right)^{r_1/r}-2$, then there exists a closed half-space $H''$ which is $r_1$-critical for $A$ and such that $\mu_B(H'')= \left\lceil \left(\frac{2|B|}{3}\right)^{r_1/r} \right\rceil -2$.
	\end{lemma}
	
	\begin{proof}
		Notice that for every direction there is a unique oriented line that defines a $r_1$-critical half-plane for $A$ on the left side of the line.  Moreover, as the direction changes, this line changes continuously, since it is defined as a minimum of several support functions of convex sets.  Therefore, we can go from $H$ to $H'$ by a continuous change of the boundary line, while always maintaining an $r_1$-critical half-plane for $A$.  The value of $\mu_B(\cdot)$ on this half-plane can only change by increments or decrements of one, as $B$ is in general position.  Therefore, at some point $\mu_B (\cdot)$ has the required value.
	\end{proof}
	
	If the conditions of the Lemma \ref{lemma-twocutting} are satisfied, define $r_2 = r - r_1$.  Then, we have a partition of $\rr^2$ into two closed convex parts $(C_1, C_2)$ such that
	\begin{align*}
		 \mu_A (C_1) \ge \left(\frac{2|A|}{3}\right)^{r_1/r} -2, \qquad &   \mu_A (C_2) \ge \left( \frac{3}{2}\right)^{(r_1/r)}|A|^{1-(r_1/r)} \ge \left(\frac{2|A|}{3}\right)^{r_2/r}-2 \\ \mu_B (C_1) \ge \left(\frac{2|B|}{3}\right)^{r_1/r} -2
		, \qquad &  \mu_B (C_2) \ge \left( \frac{3}{2}\right)^{(r_1/r)}|B|^{1-(r_1/r)} \ge \left(\frac{2|B|}{3}\right)^{r_2/r}-2
	\end{align*}
	
	This means we have an equitable $(r_1,r_2)$ cut.  Therefore, if $A, B$ are two finite sets of lines in the plane, each in general position, such that there is no equitable $(r_1, r_2)$ cut, we either have that
	\begin{itemize}
		\item Every $r_1$-critical half-plane $H$ for $A$ satisfies $\mu_B(H) > \left(\frac{2|B|}{3}\right)^{r_1/r}-2$, or
		\item Every $r_1$-critical half-plane $H$ for $A$ satisfies $\mu_B(H) < \left(\frac{2|B|}{3}\right)^{r_1/r}-2$.
	\end{itemize}
	
	We can assign a sign to $r_1$ depending on which scenario above holds true.  We will say that $r_1$ is positive for $A$ if the first one happens, and negative for $A$ otherwise.  Notice that $r_1$ is positive for $A$ if and only if it is negative for $B$.
	
	Now we can use the following theorem.
	
	\begin{theorem}[Theorem 9 in \cite{Bespamyatnikh:2000tn}]\label{theorem-summands}
		If every element of $1,2,\ldots, r-1$ is given a sign, there is either a pair $(r_1, r_2)$ or a triple $(r_1,r_2,r_3)$ with sum $r$ of the same sign.  Moreover, we can further assume that $r_i \le 2r/3$ for all $i$.
	\end{theorem}
	
	\begin{lemma}
		Let $A, B$ be two finite sets of lines in the plane, each in general position.  Let $r_1, r_2, r$ be positive integers such that $r_1 + r_2 = r$.  If there is no equitable $(r_1, r_2)$ cut, then the signs given to $r_1$ and $r_2$ are different.
	\end{lemma}
	
	\begin{proof}
		Suppose that there are such $r_1, r_2$ with the same sign and we look for a contradiction.  By swapping the role of $A$ and $B$, we may assume that both $r_1, r_2$ have positive sign.  Let $x_0$ be the value such that the half-plane $H_0 = \{(x,y) : x \le x_0\}$ is $r_1$-critical for $A$.  Let $x_1$ be the value for which the half-plane $H_1 = \{(x,y) : x \ge x_1\}$ is $r_2$-critical for $A$.  By Lemma \ref{lemma-two-parts}, we know $x_0 \le x_1$.  Therefore, any vertical line between these two half-planes induces an $(r_1, r_2)$ equitable cut, which is the contradiction we wanted.
	\end{proof}

\section{Region of convex canonical cuts}\label{section-region-construction}

For this section, consider $A$ to be a finite set of lines in general position and $r$ a positive integer.  By applying an appropriate rotation we may assume that no two points of $I(A)$ have the same $x$-coordinate.  Let $(r_1, r_2, r_3)$ be a triple of positive integers such that $r_1 + r_2 + r_3 = r$.  Let $x_0, x_1$ be the numbers such that the half-planes
\begin{align*}
	H_0 & = \{(x,y) :  x \le x_0\} \\
	H_1 & = \{(x,y) : x \ge x_1\}
\end{align*}

are $r_1$-critical and $r_2$-critical for $A$, respectively.  We know that $x_0 \le x_1$.  For each point $p=(x,y)$ such that $x_0 \le x \le x_1$ we are going to define a \textbf{canonical $p$-cutting}.  This is going to be a partition of the plane into three parts $(C_1,C_2,C_3)$.

For convenience, let $M_i = \left\lceil \left( \frac{2|A|}{3}\right)^{r_i/r}\right\rceil-2$ for $i=1,2,3$.  In order to find our partition, the main idea is to construct three rays starting from $p$.  The first ray $r_0$ is pointing downwards.  Given an angle $\alpha_1$, we define $C_1$ to be the region made by a clockwise angle of $\alpha_1$ starting at $r_0$.  We choose $\alpha_1$ to be the minimum number such that $\mu_A(C_1) = M_1$.  Notice that due to the location of $p$, we know that $\alpha \le \frac{\pi}{2}$ (i.e., $C_1$ is convex).  We define $\alpha_2, C_2$ equivalently on the other side with $\alpha_2$ now begin a counter-clockwise angle such that $\mu_A(C_2) = M_2$ and $\alpha_2$ is minimal with that property.  The region $C_3$ is the top region, which may or may not be convex.  See Figure \ref{figure-canonical}.

\begin{figure}[h]
	\centerline{\includegraphics{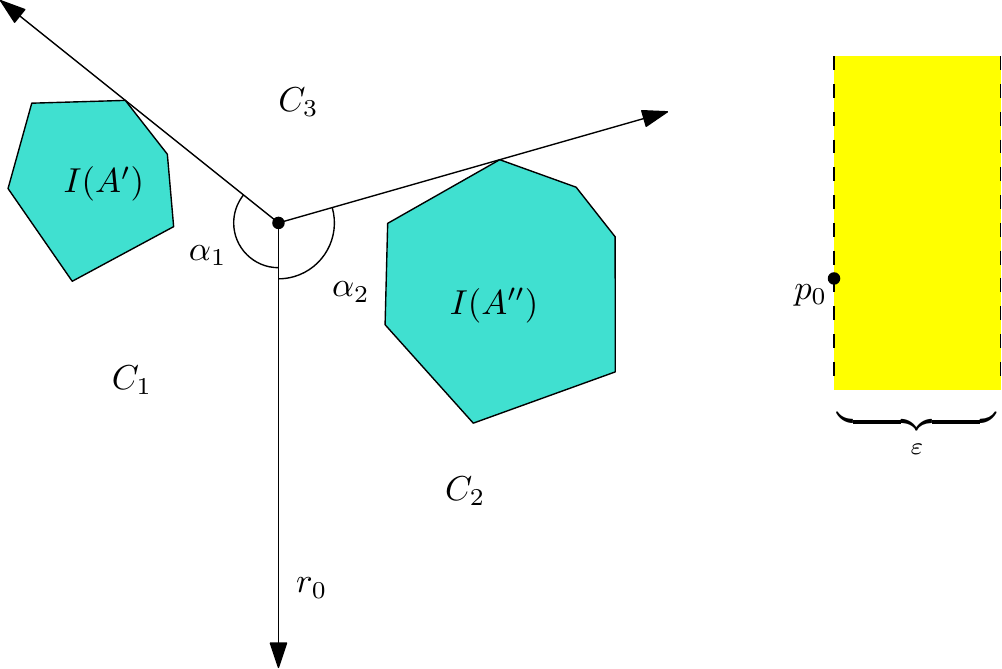}}
	\caption{Construction of the $p$-canonical cutting. For points in the vertical strip of width $\varepsilon$ on the right side of a point $p_0$ of $I(A)$, we need to redefine the value of $\alpha_1$.}\label{figure-canonical}
\end{figure}

The main issue with this construction is that $\alpha_1, \alpha_2$ are not continuous as functions of $p$, which is undesirable.  The discontinuities may occur when $p$ shares the $x$-coordinate of a point of $I(A)$.  We will refine the definitions of $\alpha_1$ and $\alpha_2$ to avoid this problem.  We will redefine the values of $\alpha_1, \alpha_2$ when the $x$-coordinate of $p$ is equal to $x_0$, when it is equal to $x_1$, and when it it very close to the $x$-coordinate of a point in $I(A)$, but not very close to $x_0$ or $x_1$.

Let $p_0 = (x',y')$ be a point of $I(A)$ such that $x_0 < x' < x_1$.  Let $\varepsilon$ be a small positive real number such that the no closed vertical strip of width $2\varepsilon$ contains more than one point of $I(A)$.  If $p = (x,y)$ and $p = (x'+t\varepsilon, y)$ for some $t \in [0,1]$, we redefine $\alpha_1$.  Let
\begin{itemize}
	\item $\alpha'_1$ be the minimum angle such that $C_1$ encloses a subset $A' \subset A$ of size $M_1$.
	\item $\alpha''_1$ be the minimum angle such that $C_1$ encloses a subset $A' \subset A$ of size $M_1$ but such that $p_0 \not\in I(A')$. 
\end{itemize}
In order to define $C_1$, we use an angle of $\alpha_1 = t\alpha'_1 + (1-t)\alpha''_1$.  Since $\alpha_1'' \ge \alpha'_1$, we know that $\alpha_1 \ge \alpha'_1$, so $\mu_A (C_1) \ge M_1$.  However, if $\mu_A(C_1) \ge M_1 + 2$, by removing at most one line from the set realizing $\mu_A(C_1)$, we can assume that $p_0$ is not part of its incidence set.  This would imply that $\alpha_1 > \alpha''_1$, which is a contradiction.  Therefore, $M_1 +1 \ge \mu_A(C_1)$.

The value of $\alpha_2$ is not changed in this region.  If $p = (x'-t\varepsilon, y)$, then we do an analogous modification by swapping the roles of $C_1, C_2$.  The reader may notice that now the angles $\alpha_1, \alpha_2$ are continuous functions of the point $p$.

Along the line $x = x_0$ and the region $x_0 \le x \le x_0 + \varepsilon$ we will redefine $\alpha_1$.  For this, assume that we have another set $B$ of lines in general position, and that $\mu_B(H_0) \ge \left\lceil \left(\frac{2|B|}{3}\right)^{r_1/r} \right\rceil - 2$.  Also assume that $I(B)$ has no points on the line $x=x_0$.

Since $H_0$ is $r_1$-critical for $A$, it means that there is a set $A_0 \subset A$ such that $H_0$ encloses $A_0$, the cardinality of $A_0$ is exactly $M_1$ and there is a unique point $p_0 = (x_0,y_0)$ in $I(A_0)$.  Notice that $\alpha_1$ is not continuous in the line $x= x_0$.  At any point $p = (x_0, y)$ with $y< y_0$, $\alpha_1 = \pi/2$.  However, at $p_0$, $\alpha_1$ defines a ray $r_1$ whose slope is equal to the slope of the top tangent of $\conv(I(A_0))$ at $p_0$ (if there are multiple sets $A_0$ that satisfy the properties above, then the slope is the minimum of the top tangents to those sets).  For a point $p = (x_0, y)$ we define

\begin{itemize}
	\item $\alpha'_1$ to be the minimum angle such that $\mu_A(C_1) \ge M_1$ and
	\item $\beta_1$ to be the minimum angle such that $\mu_B (C_1) \ge \left\lceil \left(\frac{2|B|}{3}\right)^{r_1/r} \right\rceil-2$. 
\end{itemize}

Now we define the angle $\tilde{a}_1 = \tilde{a}_1(x_0, y)$ of $C$ as
\[
\tilde{a}_1 = \begin{cases}
	\pi / 2 & \mbox{ if } y < y_0 \\
	t \max \{\alpha'_1, \beta_1 \} + (1-t) \pi / 2 & \mbox{ if } y = y_0 + t \varepsilon, \qquad t \in [0,1] \\
	\max \{\alpha'_1, \beta_1 \} & \mbox{ if } y > y_0 + \varepsilon
\end{cases}
\]

The region $C_1$ defined by angle $\tilde{\alpha}_1$ satisfies $\mu_A(C_1) = M_1$ (since $H_0$ was $r_1$-critical) and $\mu_B (C_1) \ge \left\lceil \left(\frac{2|B|}{3}\right)^{r_1/r} \right\rceil-2$.

Then, for $p = (x_0 + s\varepsilon, y)$ for some $s \in [0,1]$, we define
\begin{itemize}
	\item $\alpha''_1$ the minimum angle such that the region $C_1$ satisfies $\mu_A(C_1) \ge M_1$.
	\item $\alpha_1 = s\alpha''_1 + (1-s)\tilde{\alpha}_1(x_0,y)$.
\end{itemize}

The angle $\alpha_1$ is now a continuous function, $M_1 + 1 \ge \mu_A (C_1) \ge M_1$ and the behavior of $C_1$ on the line $x = x_0$ is the one we described above. We do an analogous definition for $\alpha_2$.  Now we are ready to define the region of the points $p$ we are interested in.

Let
\[
R = \{p=(x,y) \in \rr^2 : x_0 \le x \le x_1 \mbox{ and the region } C_3 \mbox{ of the canonical $p$-cutting is convex}\}.
\]

Notice that the top boundary of $R$ is defined by the equation $\alpha_1 + \alpha_2 \ge \pi/2$.  The region $R$ is bounded above and unbounded below.  For every point on the top boundary the region $C_3$ is a half-plane.  Moreover, the continuity of $\alpha_1$ and $\alpha_2$ implies that there are no vertical segments in the boundary of $R$ except for those contained on the lines $x= x_0$ and $x = x_1$.

\begin{figure}[h]\label{figure-region-R}
	\centerline{\includegraphics{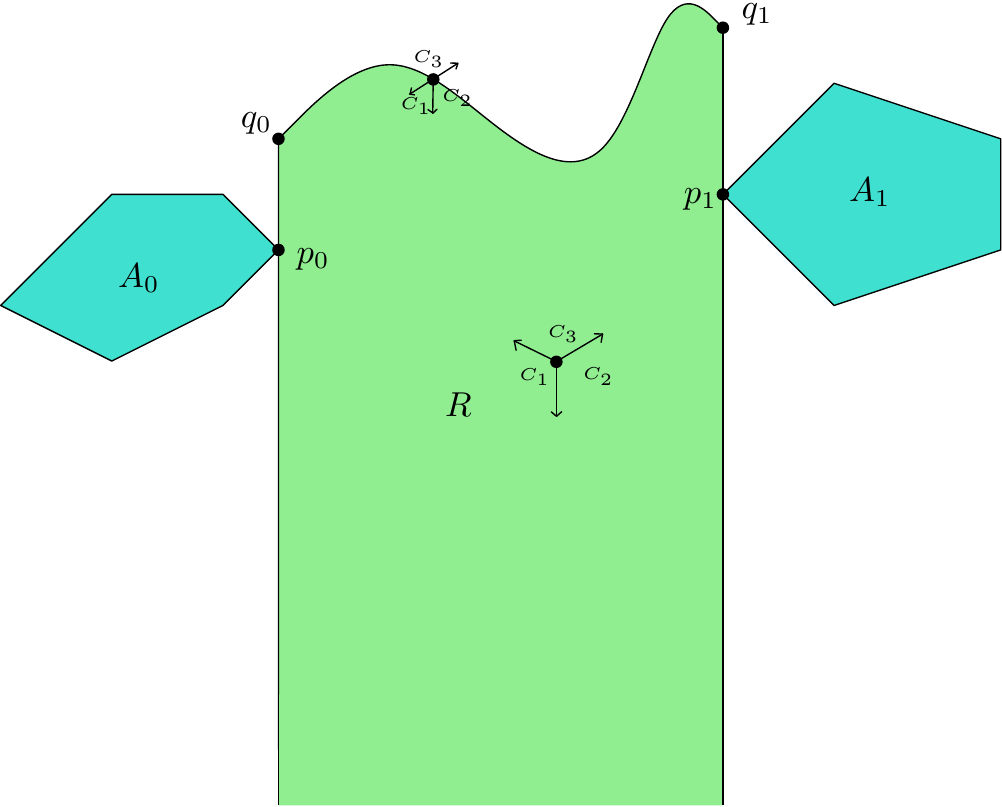}}
	\caption{An example of the region $R$, with a couple of canonical $p$-cuttings marked.}
\end{figure}



\begin{claim}\label{claim-goodpartition}
	For every point $p$ of $R$, we have $\mu_A(C_3) \ge M_3.$
\end{claim}

\begin{proof}
	If we apply Lemma \ref{lemma-xue2}, we know that $\mu_A(C_1) \mu_A(C_2) \mu_A(C_3) \ge \frac{2|A|}{3}$.  We also know that $(M_1 +1)( M_2 +1)( M_3+1) < \frac{2|A|}{3}$.  Therefore, since $\mu_A (C_1) \le M_1 +1$ and $\mu_A (C_2) \le M_2 +1$, we have $\mu_A (C_3) \ge M_3$.\end{proof}

\section{Proof of Lemma \ref{lemma-strong-equitable-cuts}}\label{section-big-proof}

With the construction of the region $R$ in the previous section, we are ready to prove Lemma \ref{lemma-strong-equitable-cuts}.

\begin{proof}[Proof of Lemma \ref{lemma-strong-equitable-cuts}]
	Let $A, B$ be two sets of lines in $\rr^2$, each in general position.  We may assume that no two points of $I(A) \cup I(B)$ share an $x$-coordinate.  If there is no $(r_1, r_2)$-equitable cut for any pair $(r_1, r_2)$ of positive integers with sum $r$, we may assign to every integer in $\{1,\ldots, r-1\}$ a sign as in Section \ref{section-equitable-cuttings}.  By Theorem \ref{theorem-summands}, there exists a triple $(r_1, r_2, r_3)$ of positive integers whose sum is $r$ and each $r_i$ has the same sign.  By swapping the roles of $A$ and $B$, we can assume that the sign of each $r_i$ is positive.
	
	For each $i \in \{1,2,3\}$, let $N_i = \left( \frac{2|B|}{3}\right)^{r_i/r}-2$.
	
	We are going to color the points of $R$, where each color is represented by an element of $\{1,2,3\}$.  The point $p \in R$ is going to be colored of color $i$ if and only if in the canonical cutting of $p$, we have $\mu_B (C_i) \ge N_i$.  Points are allowed to have multiple colors.  By Claim \ref{claim-goodpartition}, it is sufficient to prove that a point $p$ has all three colors to finish the proof.  Notice the following lemma.
	
\begin{claim}
	Every point of $R$ has at least one color.
\end{claim}

\begin{proof}
This follows from Lemma \ref{lemma-xue2} and the fact that $N_1 N_2 N_3 < \frac{2|B|}{3}.$
	\end{proof}
	
	\begin{claim}
		Every point of the top boundary of $R$ has color $3$, every point of $R$ on the line $x=x_0$ has color $1$ and every point of $R$ on the line $x = x_1$ has color $2$.
	\end{claim}
	
	\begin{proof}
	Since $r_1, r_2, r_3$ all have positive sign, we know that $\mu_B(H_0) \ge N_1$ and $\mu_B(H_1) \ge N_2$.  By the definition of the canonical cuttings on the lines $x=x_0$ and $x=x_1$, we have the claim on the left and right boundaries.  For every point $p$ on the top boundary, we know that $C_3$ is a half-plane with $\mu_A(C_3) \ge M_3$.  This means that $C_3$ contains a $r_3$-critical halfplane $C'_3$.  Therefore, $\mu_B (C_3) \ge \mu_B(C'_3) \ge N_3$, due to the sign of $r_3$. 
	\end{proof}
	
	\begin{claim}
	The color classes are closed sets.	
	\end{claim}
	
	\begin{proof}
		Take any converging  sequence of points of color $i$.  There must be a subsequence on which $C_i$ encloses the same subsets of $B$.  By the continuity of the angles $\alpha_1, \alpha_2$, on the point of convergence we also have that $C_i$ encloses those subsets.
	\end{proof}
	
	\begin{claim}
	There is a sufficiently small value $\tilde{y}$ such that every point on the intersection of the line $y = \tilde{y}$ and the region $R$ has colors $1$ and $2$.	
	\end{claim}
	
	\begin{proof}
	For every value of $x$, there has to be a $y_x$ such that if $y < y_x$, then $C_1$ encloses $A_0$ for the canonical cutting of $p=(x,y)$.  Moreover, the value of $y_x$ is a continuous function of $x$.  Therefore, it attains a minimum value $\tilde{y}_1$ on $[x_0, x_1]$.  Points in $R$ with $y$-coordinate $\tilde{y}_1$ or less will have color 1. Similarly, we can find a $\tilde{y}_2$ such that points in $R$ with $y$-coordinate $\tilde{y}_2$ or less will have color 2.
	\end{proof}
	
	We define $R' = \{p=(x,y) : p \in R, \ y \ge \tilde{y} \}$.  Now, we are able to apply the classic Knaster-Kuratowski-Mazurkiewicz (KKM) theorem in dimension two.
	
	\begin{theoremp}[Knaster, Kuratowski, Mazurkiewicz \cite{Knaster:1929vi}]
	Let $\Delta$ be a triangle with vertices $1, 2, 3$.  Suppose that $\Delta$ is colored with colors $\{1,2,3\}$ such that every vertex $i$ has color $i$, and every point on a side $ij$ has at least one of the colors $i$ or $j$.  If every color class is a closed set, then there is a point with all three colors.
	\end{theoremp}
	
	We can choose a point on the left side of the boundary of $R'$ to be vertex $1$, a point on the right side of the boundary of $R'$ to be vertex $2$, and a point on the top boundary to be vertex $3$.  The application of the KKM theorem finishes the proof.

	\end{proof}

\section{Upper bounds}\label{section-upper-bounds}

In order to obtain upper bounds for our results, let's start with a single set of lines that is hard to split using vertical strips.  Consider $[a] = \{1,\ldots, a\}$.

\begin{theorem}\label{theorem-upper-bound-vertical}
	Let $a$ and $r$ be positive integers.  There is a set $A$ of $a^r$ lines in the plane in general position such that for any partition of the plane into $r$ closed vertical strips $A_1, \ldots, A_r$, there exists an $i \in [r]$ such that
	\[
	\mu_A(A_i) \le a.
	\] 
\end{theorem}

\begin{proof}
Let us first construct a set $X$ of $a^r$ points in the plane.  We consider $v_1, \ldots, v_r$ nonzero vectors in $\rr^2$, none of which are vertical or horizontal, with the following properties
\begin{itemize}
	\item The slope of $v_i$ is $l_i$, and $l_1 < l_2 < \ldots < l_r$
	\item The norms of the vectors satisfy $||v_1|| < \ldots < ||v_r||$.  Moreover, $||v_{i+1}||$ is significantly larger than $||v_i||$, in a way that will be made precise below.
\end{itemize}

Then, let
\[
X = \left\{\sum_{i=1}^r x_i v_i : (x_1, \ldots, x_r) \in [a]^r\right\}.
\]
Let $\varepsilon >0$ be a sufficiently small number such that the intervals $I_i = [l_i-\varepsilon,l_i+\varepsilon]$ are pairwise disjoint.  For two points $p=\sum_{i=1}^r x_i v_i$ and $q=\sum_{i=1}^r y_i v_i$ in $X$, we consider
\[
m(p,q) = \max\{ i \in [a] : x_i \neq y_i\}.
\]
We want the sequence $||v_1||, \ldots, ||v_r||$ to grow fast enough so that, if $j = m(p,q)$, then the slope of $p-q$ lies in the interval $[l_j - \varepsilon, l_j + \varepsilon] = I_j$.

Now we construct the set $A$ of lines (given by their equations) as
\[
A = \{y=mx+c : (m,c) \in X\}
\]

If we are given two lines with equations $y=m_1 x + c_1$ and $y=m_2 x + c_2$, the $x$-coordinate of their intersection is given by $x = \frac{c_2 - c_1}{m_1 - m_2}$, which is the negative of the slope between $(m_1, c_1)$ and $(m_2,c_2)$.  Now, consider a partition of $\rr^2$ into $r$ vertical strips.  The boundary of these strips is given by the $r-1$ lines $x=t_1, x=t_2, \ldots, x=t_{r-1}$, for some real $t_1 \le \ldots \le t_{r-1}$.  A simple inductive argument shows that at least one of the intervals $[-\infty, t_1], [t_1, t_2], \ldots , [t_{r-1}, \infty]$ intersects at most one of the intervals $-I_1, \ldots, -I_r$.  Suppose that $[t_i, t_{i+1}]$ intersects only $-I_j$.  

This means that if a subset $A' \subset A$ is enclosed by the strip between the lines $x=t_i, x=t_{i+1}$, these lines came from a subset $X' \subset X$ whose pairwise slopes are contained in $I_j$.  This, in turn, implies that for $p,q \in X'$ we have that $m(p,q) = j$, so all the points of $X'$ differ in the $j$-th coordinate (as vectors in $[a]^r$).  Thus, $|A'|=|X'| \le a$, as we wanted.  If we want our set of lines to be in general position, a small perturbation of $A$ gives us the desired set.
\end{proof}

Now we are ready to prove Theorem \ref{theorem-upper-bound}.

\begin{proof}[Proof of Theorem \ref{theorem-upper-bound}]
	We may assume without loss of generality that $|A|$ is the $r$-th power of an integer.   We set $A$ to be the example from Theorem \ref{theorem-upper-bound-vertical}.  We set $B$ to be such that $I(B)$ is contained in a disk of small radius.  If we start to translate the set $B$ upwards, and $C_1, C_2, \ldots, C_r$ is a convex partition of the plane where each $C_i$ intersects both $\conv I(B)$ and $\conv I(A)$, then the boundary between all pairs $C_i, C_j$ is either above $\conv(B)$ or very close to a vertical line.  Therefore, for a sufficiently high $I(B)$, the properties that the set $A$ satisfies complete the proof.
\end{proof}

\section{Remarks}\label{section-remarks}

The proof of our main results follows the ideas from Bespamyatnikh, Kirkpatrick, and Snoeyink \cite{Bespamyatnikh:2000tn}.  It would be interesting to see if it is possible to obtain a proof that uses stronger topological tools, as we have for point sets \cite{Soberon:2012kp, Karasev:2014gi, Blagojevic:2014ey}.  The problem boils down to the following question, which would be a nice extension of the Erd\H{o}s-Szekeres theorem.

\begin{question}
	Determine if the following statement is true.  Given a finite set $A$ of lines in the plane and a convex partition $(C_1, \ldots, C_r )$ of the plane that comes from a power diagram, the following equation holds
	\[
	\mu_A(C_1)\mu_A(C_2) \ldots \mu_A(C_r) \ge |A|.
	\]
\end{question}

Actually, it would be sufficient to have 
\[
\max_{1\le i\le r} \mu_A(C_i) \ge |A|^{1/r}.
\]

One may also wonder what happens in higher dimensions.  We say that a set of hyperplanes $A$ in $\rr^d$ is in general position if every $d$ of its normal vectors are linearly independent and no $d+1$ hyperplanes of $A$ share a point.  We denote by $I(A)$ the set of all points that come from the intersection of $d$ hyperplanes in $A$, and say that $K$ encloses $A$ if $I(A) \subset K$.  It is natural to ask the following question.

\begin{question}
	Given $n, d, r$ positive integers, find the smallest value of $M=f_{d,r}(n)$ such that the following holds.  Given $d$ sets $A_1, \ldots, A_d$ of $M$ hyperplanes each in $\rr^d$, there exists a convex partition of $\rr^d$ into $r$ parts $(C_1, \ldots, C_r)$ such that each $C_j$ encloses a subset of at least $n$ hyperplanes of each $A_i$.
\end{question}

Dujmovi\'c and Langerman proved the existence of such a function when $r=2$, and the rate of growth of $f_{d,2}$ has been bounded by Conlon, Fox, Pach, Sudakov, and Suk \cite{Conlon:2014fx} to
\[
f_{d,2}( n) \le \operatorname{twr}_{d-1} (c_dn^2 \log n),
\]

where the tower function $\operatorname{twr}_{d-1}(\cdot)$ is the composition of the function $2^x$ with itself $d-1$ times, and $c_d$ is a constant that depends only on $d$.

If one is interested to see if the constant leading constant $r^{\ln(2/3)}$ can be replaced by $1$ with the current proof, we would need to remove the need for $2/3$ in Lemma \ref{lemma-xue} or in Lemma \ref{lemma-xue2}.  However, this replacement cannot be done for Lemma \ref{lemma-xue}. Figure \ref{figure-five-lines} shows an example of five lines and two rays such that none of the two sides of the broken line encloses more than two lines.

\begin{figure}
\centerline{\includegraphics[scale=0.5]{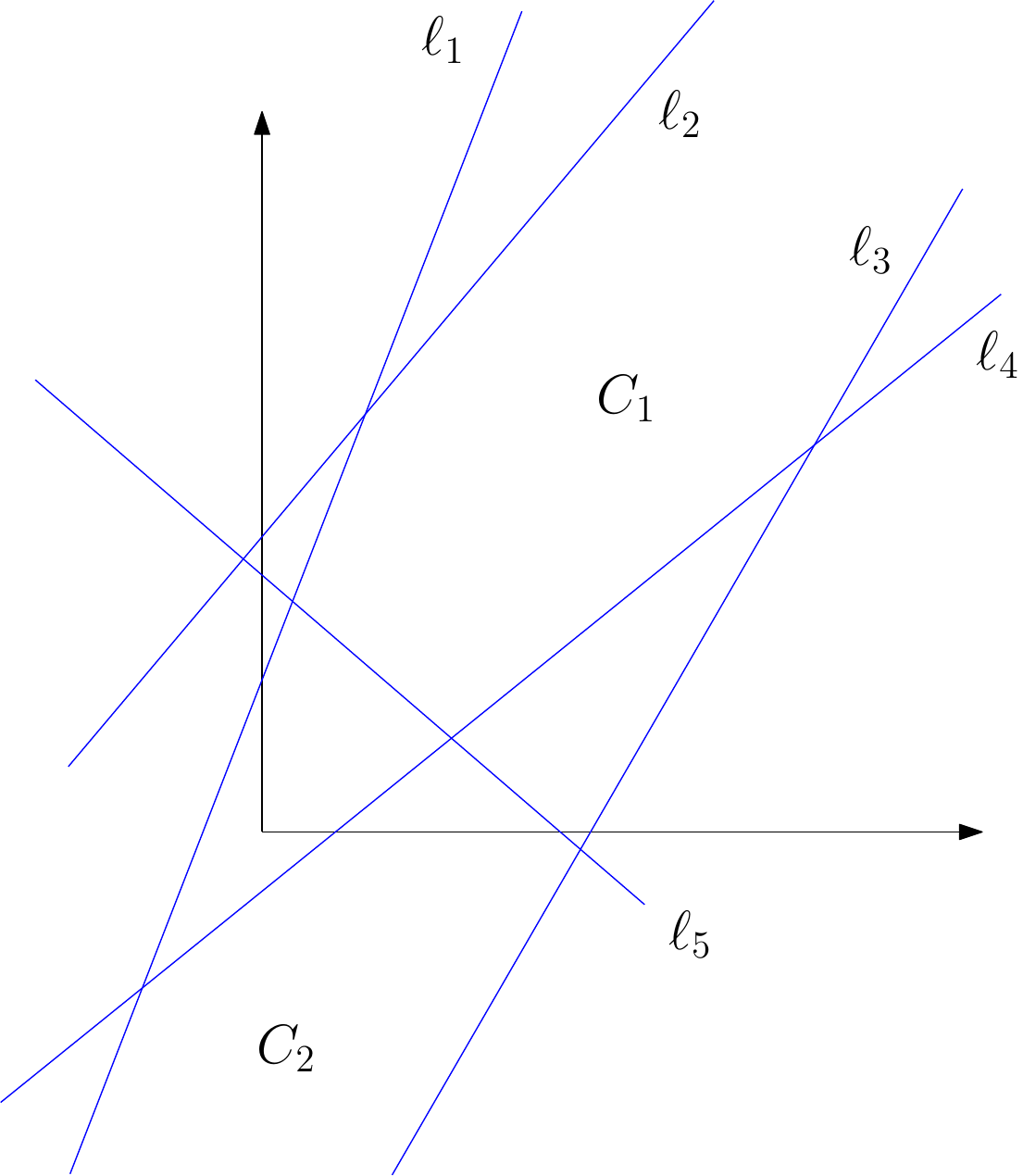}}	
\caption{Every triangle made by three blue lines intersects the broken black line.  The points $\ell_1 \cap \ell_3$ and $\ell_2 \cap \ell_4$ are in $C_2$.}
\label{figure-five-lines}
\end{figure}

The loss of $2r$ lines in Theorem \ref{theorem-main} comes from requiring that the angles of the canonical $3$-cuttings are continuous.  Without this assumption, we are not able to guarantee that the coloring of the region $R'$ satisfies the properties of the KKM theorem.  Moreover, it would also allow the top boundary of $R$ to have vertical segments, which makes the analysis more difficult (the arguments presented would only show that the top point of each vertical segment has color $3$, instead of the whole segment).


\newcommand{\etalchar}[1]{$^{#1}$}
\providecommand{\bysame}{\leavevmode\hbox to3em{\hrulefill}\thinspace}
\providecommand{\MR}{\relax\ifhmode\unskip\space\fi MR }
\providecommand{\MRhref}[2]{%
  \href{http://www.ams.org/mathscinet-getitem?mr=#1}{#2}
}
\providecommand{\href}[2]{#2}

\end{document}